\documentclass[11pt]{amsart}

\usepackage[T1]{fontenc}
\usepackage{lmodern}
\usepackage{amsmath,amsthm,amssymb}
\usepackage{booktabs,array,tikz,placeins,microtype}
\usepackage[colorlinks,linkcolor=blue,citecolor=green,
urlcolor=blue]{hyperref}
\usepackage[capitalise,noabbrev]{cleveref}
\hypersetup{pdftitle={Weighted singular vectors in common-base self-similar sets},
pdfauthor={Taehyeong Kim and Jaemin Park}}

\numberwithin{equation}{section}

\newtheorem{theorem}{Theorem}[section]
\newtheorem{corollary}[theorem]{Corollary}
\newtheorem{lemma}[theorem]{Lemma}
\newtheorem{proposition}[theorem]{Proposition}
\theoremstyle{definition}
\newtheorem{definition}[theorem]{Definition}
\theoremstyle{definition}
\newtheorem{remark}[theorem]{Remark}

\title[Weighted singular vectors in self-similar sets]
{Weighted singular vectors in common-base self-similar sets}
\author{Taehyeong Kim}
\address{Department of Mathematics, Brandeis University, 415 South Street,
Waltham, MA 02453, USA}
\email{taehyeongkim@brandeis.edu}
\author{Jaemin Park}
\address{Department of Mathematics and Research Institute of Natural Sciences,
Sookmyung Women's University, Cheongpa-ro 47-gil 100, Yongsan-ku,
Seoul 04310, Korea}
\email{jaemin.park@sookmyung.ac.kr}
\date{}
\subjclass[2020]{Primary 11J13; Secondary 11J83, 28A80}

\begin{document}

\begin{abstract}
We prove a lower bound for the Hausdorff dimension of weighted totally
irrational singular vectors in affine-spanning common-base integral
self-similar sets satisfying the open set condition. For the middle-third
Cantor square, the bound improves the previously known explicit lower bound.
\end{abstract}
\maketitle

\section{Introduction}\label{introduction-and-main-results}

Dirichlet's theorem (1842), one of the foundational results in Diophantine
approximation, asserts that, for every \(x\in\mathbb R^n\) and every \(T>1\),
there are \(q\in\mathbb Z\) and \(\mathbf p\in\mathbb Z^n\) such that
\begin{equation}\label{eq:dirichlet}
0<q\le T,
\qquad
\|qx-\mathbf p\|\le T^{-1/n}.
\end{equation}
Here \(\|y\|:=\max_{1\le i\le n}|y_i|\) denotes the max norm on
\(\mathbb R^n\).
The vector \(x\) is \emph{singular} if Dirichlet's inequality can be improved
by an arbitrary factor: for every \(\varepsilon>0\), the same conclusion as in
\eqref{eq:dirichlet} holds with \(\varepsilon T^{-1/n}\) in place of
\(T^{-1/n}\) for all sufficiently large \(T\). We denote the set of singular
vectors in \(\mathbb R^n\) by \(\operatorname{Sing}(n)\). Khintchine
introduced singular vectors in \cite{Khi26}. In dimension one, the singular
vectors are precisely the rational numbers, so the dimension problem is
trivial; it becomes nontrivial in higher dimension. We therefore assume
\(n\ge2\) throughout. With \(\dim_H\) denoting Hausdorff dimension, Cheung
\cite{Che11} proved that
\(\dim_H\operatorname{Sing}(2)=4/3\), and Cheung--Chevallier \cite{CC16}
proved that \(\dim_H\operatorname{Sing}(n)=n^2/(n+1)\) for every \(n\ge2\).
The corresponding matrix problems were studied by Kadyrov et al.\
\cite{KKLM17} and Das et al.\ \cite{DFSU24}. We restrict attention to vectors.

Weighted singularity allows different approximation rates in the
coordinates. For a weight
\(\mathbf w=(w_1,\ldots,w_n)\), where
\(1>w_1\ge\cdots\ge w_n>0\) and \(\sum_{i=1}^n w_i=1\),
a vector \(x\in\mathbb R^n\) is \emph{\(\mathbf w\)-singular} if, for every
\(\varepsilon>0\), there is \(T_0>1\) such that for every \(T>T_0\) there
are \(q\in\mathbb Z\) and \(\mathbf p\in\mathbb Z^n\) satisfying
\begin{equation}\label{eq:weighted-singularity}
0<q<T,
\qquad
\max_{1\le i\le n}|qx_i-p_i|^{1/w_i}<\frac{\varepsilon}{T}.
\end{equation}
We denote the set of such vectors by \(\operatorname{Sing}(\mathbf w)\).
When \(w_i=1/n\) for every \(i\), this definition recovers
\(\operatorname{Sing}(n)\).
In dimension two, Liao--Shi--Solan--Tamam \cite{LSST20} proved
\(\dim_H\operatorname{Sing}(\mathbf w)=2-1/(1+w_1)\). For general \(n\),
the authors \cite{KP24} obtained the lower bound
\(\dim_H\operatorname{Sing}(\mathbf w)\ge n-1/(1+w_1)\).
A vector \(x=(x_1,\ldots,x_n)\in\mathbb R^n\) is called
\emph{totally irrational} if
$1,x_1,\ldots,x_n$ are linearly independent over $\mathbb Q$.
We set
\[
\operatorname{Sing}^*(\mathbf w)
:=\{x\in\operatorname{Sing}(\mathbf w):x\text{ is totally irrational}\}.
\]
For equal weights we use the abbreviation \(\operatorname{Sing}^*(n)\).

Singular vectors have also been studied on submanifolds and fractals.
Kleinbock--Moshchevitin--Weiss \cite{KMW21} constructed totally irrational
singular vectors with large uniform exponents on broad classes of analytic
manifolds and fractals, while Shah--Yang \cite{SY24} obtained
Hausdorff-dimension upper bounds on affine subspaces in terms of their
Diophantine exponents. In the weighted setting, Datta--Tamam \cite{DT24}
proved inheritance and existence results, and
Kleinbock--Moshchevitin--Warren--Weiss \cite{KMWW25} established strong
intersection properties and quantitative rates for uniform approximation.

A basic unresolved dimension problem in the fractal setting already appears
for the middle-third Cantor square. Let \(C\) denote the middle-third Cantor
set, and write \(C^2=C\times C\). Bugeaud--Cheung--Chevallier
\cite[Problem~6]{BCC19} asked for the
Hausdorff dimension of
\[
\operatorname{Sing}(2)\cap C^2.
\]
To the best of our knowledge, its exact value remains unknown. Khalil's upper
bound \cite{Kha20}, specialized to \(C^2\), gives
\[
\dim_H\bigl(\operatorname{Sing}(2)\cap C^2\bigr)
\le
\frac23\dim_H(C^2)
=
\frac{4\log2}{3\log3}.
\]
In the weighted setting, Aggarwal--Ghosh
\cite[Theorem~1.3, case~(2), and Remark~1.4]{AG26} give the upper bound
\[
\dim_H\bigl(\operatorname{Sing}(\mathbf w)\cap C^2\bigr)
\le
\frac{2w_1}{1+w_1}\dim_H(C^2).
\]
\begin{samepage}
On the lower-bound side, Schleischitz \cite[Theorem~4.2]{Sch22} obtained
\[
\dim_H\bigl(\operatorname{Sing}^*(2)\cap C^2\bigr)>0.1255.
\]
The following theorem gives a weighted lower bound and, in the unweighted
case, improves this estimate.
\end{samepage}

\begin{theorem}\label[theorem]{thm:cantor-square}
Let \(C\) be the middle-third Cantor set and let
\(\mathbf w=(w_1,w_2)\), where
\(1>w_1\ge w_2>0\) and \(w_1+w_2=1\). Then
\[
\dim_H\bigl(\operatorname{Sing}^*(\mathbf w)\cap C^2\bigr)
\ge
\left(\frac{1-w_1}{1+w_1}\right)^2\dim_H(C^2)
=
2\left(\frac{1-w_1}{1+w_1}\right)^2\frac{\log2}{\log3}.
\]
In particular,
\[
\dim_H\bigl(\operatorname{Sing}^*(2)\cap C^2\bigr)
\ge \frac{2\log2}{9\log3}.
\]
\end{theorem}

Since \(2\log2/(9\log3)\approx0.1402\), \cref{thm:cantor-square} improves
the explicit estimate of Schleischitz. Beyond the Cantor square, these bounds
extend in different directions: Khalil's upper-bound theorem applies to
self-similar fractals satisfying the open set condition (OSC), the weighted
upper bound of Aggarwal--Ghosh applies to Cartesian products of homogeneous
one-dimensional self-similar sets satisfying OSC, and Schleischitz's lower
bounds apply to Cartesian products of common-base missing-digit sets. The
construction behind
\cref{thm:cantor-square} yields the following lower bound for common-base
integral self-similar sets.

For \(A\subset\mathbb R^n\), write \(\operatorname{aff}_{\mathbb R}(A)\) for
the smallest affine subspace of \(\mathbb R^n\) containing \(A\). For a
common-base integral self-similar set \(K\) with digit set
\(\mathcal D\subset\mathbb Z^n\) as in \cref{def:common-base-fractal}, if
\(\operatorname{aff}_{\mathbb R}(\mathcal D)\ne\mathbb R^n\), then \(K\) is
contained in a rational affine hyperplane and hence
\(K\subset\operatorname{Sing}(\mathbf w)\). This case is trivial, so we
restrict to
\(\operatorname{aff}_{\mathbb R}(\mathcal D)=\mathbb R^n\).

\begin{theorem}\label[theorem]{thm:intro-common-base}
Let \(K\subset\mathbb R^n\) be a common-base integral self-similar set with
digit set \(\mathcal D\subset\mathbb Z^n\) as in
\cref{def:common-base-fractal}, and suppose that its defining system satisfies
the open set condition as in \cref{def:osc}. Assume that
\(\operatorname{aff}_{\mathbb R}(\mathcal D)=\mathbb R^n\). Then
\[
\dim_H\bigl(\operatorname{Sing}^*(\mathbf w)\cap K\bigr)
\ge
\left(\frac{1-w_1}{1+w_1}\right)^2\dim_HK.
\]
\end{theorem}

The relation among these vector-case dimension results is summarized in
\cref{tab:related-work}.

\begin{table}[ht]
\centering
\footnotesize
\begin{tabular}{@{}
  >{\raggedright\arraybackslash}p{0.17\textwidth}
  >{\raggedright\arraybackslash}p{0.29\textwidth}
  >{\raggedright\arraybackslash}p{0.48\textwidth}@{}}
\toprule
Reference & Setting & Conclusion \\
\midrule
\cite{Che11}
& Ambient \(\mathbb R^2\), unweighted
& \(\dim_H\operatorname{Sing}(2)=4/3\) \\
\specialrule{0.25pt}{1.5pt}{1.5pt}
\cite{CC16}
& Ambient \(\mathbb R^n\), unweighted
& \(\dim_H\operatorname{Sing}(n)=n^2/(n+1)\) \\
\specialrule{0.25pt}{1.5pt}{1.5pt}
\cite{LSST20}
& Ambient \(\mathbb R^2\), weighted
& \(\dim_H\operatorname{Sing}(\mathbf w)=2-1/(1+w_1)\) \\
\specialrule{0.25pt}{1.5pt}{1.5pt}
\cite{KP24}
& Ambient \(\mathbb R^n\), weighted
& \(\dim_H\operatorname{Sing}(\mathbf w)\ge n-1/(1+w_1)\) \\
\specialrule{0.25pt}{1.5pt}{1.5pt}
\cite{Kha20}
& Self-similar fractals, unweighted
& Hausdorff-dimension upper bound\newline
  \(\dim_H(\operatorname{Sing}(2)\cap C^2)
  \le \frac23\dim_H(C^2)\) \\
\specialrule{0.25pt}{1.5pt}{1.5pt}
\cite{AG26}
& Homogeneous self-similar products, weighted
& Hausdorff-dimension upper bound\newline
  \(\dim_H(\operatorname{Sing}(\mathbf w)\cap C^2)
  \le \frac{2w_1}{1+w_1}\dim_H(C^2)\) \\
\specialrule{0.25pt}{1.5pt}{1.5pt}
\cite{Sch22}
& Common-base missing-digit products, unweighted
& Hausdorff-dimension lower bound\newline
  \(\dim_H(\operatorname{Sing}^*(2)\cap C^2)>0.1255\) \\
\specialrule{0.25pt}{1.5pt}{1.5pt}
This paper, \cref{thm:cantor-square}
& Middle-third Cantor square, weighted
& \(\dim_H(\operatorname{Sing}^*(\mathbf w)\cap C^2)
   \ge ((1-w_1)/(1+w_1))^2\dim_H(C^2)\) \\
\specialrule{0.25pt}{1.5pt}{1.5pt}
This paper, \cref{thm:intro-common-base}
& Affine-spanning common-base integral self-similar sets, weighted
& \(\dim_H(\operatorname{Sing}^*(\mathbf w)\cap K)
   \ge ((1-w_1)/(1+w_1))^2\dim_HK\) \\
\bottomrule
\end{tabular}
\vspace{8pt}
\caption{Selected Hausdorff-dimension results in the vector case.}
\label{tab:related-work}
\end{table}

We next briefly describe the main idea of the proof. In the base-\(b\) coding
of \(K\), we require all digits within each of a sequence of disjoint blocks
to equal one fixed digit. Each block gives the required rational approximation
for a range of values of the parameter \(T\) in
\eqref{eq:weighted-singularity}. The blocks
are arranged so that these ranges overlap and cover all sufficiently large
\(T\), making every point in the resulting set \(\mathbf w\)-singular. The
digits outside the prescribed blocks remain free. Under OSC, their lower
density yields the Hausdorff-dimension estimate used below; optimizing the
block lengths and positions gives the stated lower bound. Finally, the full
affine-span assumption ensures that rational affine hyperplanes have zero
measure for this construction, so the same bound holds after restricting to
totally irrational points.

The paper is organized as follows. In
\cref{common-base-fractals-and-digit-freezing}, we develop the common-base
setting, prove a dimension lower bound for sets obtained by prescribing digits,
and establish that the associated measures, when there are infinitely many
free positions, give zero mass to affine hyperplanes.
\Cref{constant-symbol-block-approximation} derives the approximation supplied
by a constant-symbol block. The block construction is carried out in
\cref{block-construction}, and the required parameter estimates are
proved in \cref{parameter-estimates}.

\FloatBarrier

\medskip
\noindent\textbf{Use of artificial intelligence.}
OpenAI Codex was used as an auxiliary tool in checking selected calculations
and arguments and in refining the exposition. The authors independently
verified every mathematical argument, determined the final content and wording,
and take full responsibility for the manuscript.

\section{Common-base fractals and digit-freezing measures}\label{common-base-fractals-and-digit-freezing}

This section introduces the common-base self-similar setting. We then derive a
uniform cylinder estimate from OSC and use it to obtain a Hausdorff-dimension
lower bound for sets obtained by prescribing digits. Under the full affine-span
hypothesis, we also show that the associated measures give zero mass to affine
hyperplanes provided that infinitely many positions remain free.

\begin{definition}[Common-base integral self-similar sets and cylinders]
\label[definition]{def:common-base-fractal}
Fix an integer \(b\ge2\) and a finite digit set
\(\mathcal D\subset\mathbb Z^n\) with \(\#\mathcal D\ge2\). The maps
\[
F_{\mathbf a}(x)=\frac{x+\mathbf a}{b},
\qquad \mathbf a\in\mathcal D,
\]
form a \emph{common-base integral self-similar system}. This is a special case
of Hutchinson's general self-similar construction in which all similarities
have the same linear part \(b^{-1}I\) and the digit vectors lie in
\(\mathbb Z^n\) \cite{Hut81}. The unique nonempty
compact set \(K\) satisfying
\[
K=\bigcup_{\mathbf a\in\mathcal D}F_{\mathbf a}(K)
\]
is its \emph{common-base integral self-similar set}.
\par\pagebreak
For a word
\(\xi=(\xi_1,\ldots,\xi_N)\in\mathcal D^N\), let
\[
F_\xi:=F_{\xi_1}\circ\cdots\circ F_{\xi_N},
\qquad
K_\xi:=F_\xi(K),
\]
and call \(K_\xi\) the corresponding level-\(N\) cylinder.
\end{definition}
\pagebreak[2]

The associated coding map is the surjection
\[
\pi:\mathcal D^{\mathbb N}\longrightarrow K,
\qquad
\pi(\eta):=\sum_{r=1}^{\infty}\eta_rb^{-r},
\]
so that \(K=\pi(\mathcal D^{\mathbb N})\). For every
\(\xi=(\xi_1,\ldots,\xi_N)\in\mathcal D^N\) and \(y\in K\), direct iteration
gives
\begin{equation}\label{eq:coding-identity}
F_\xi(y)=\sum_{r=1}^N\xi_rb^{-r}+b^{-N}y,
\qquad
K_\xi=\{\pi(\eta):\eta|_N=\xi\},
\end{equation}
where \(\eta|_N=(\eta_1,\ldots,\eta_N)\).

\begin{definition}[Open set condition]\label[definition]{def:osc}
The system in \cref{def:common-base-fractal} satisfies the
\emph{open set condition} (OSC) if there is a nonempty open set
\(O\subset\mathbb R^n\) such that
\[
F_{\mathbf a}(O)\subset O
\quad\text{for every }\mathbf a\in\mathcal D,
\qquad
F_{\mathbf a}(O)\cap F_{\mathbf a'}(O)=\varnothing
\quad\text{if }\mathbf a\ne\mathbf a'.
\]
\end{definition}

\begin{remark}
\label[remark]{rem:osc-dimension}
If the common-base system satisfies OSC, then Hutchinson's dimension theorem
\cite[Theorem~5.3]{Hut81} gives
\[
\dim_HK=\frac{\log(\#\mathcal D)}{\log b}.
\]
Equivalently, \(\#\mathcal D=b^{\dim_HK}\).
\end{remark}

The next lemma records a uniform cylinder-counting estimate under OSC.
Throughout, \(B(x,r)\) denotes the max-norm ball of radius \(r\) centered at
\(x\).

\begin{samepage}
\begin{lemma}\label[lemma]{lem:osc-counting}
Suppose that the common-base system satisfies OSC. Then there is
\(C_K<\infty\) such that
\[
\sup_{N\ge1}\sup_{x\in\mathbb R^n}
\#\left\{
\xi\in\mathcal D^N:
K_\xi\cap B(x,b^{-N})\ne\varnothing
\right\}
\le C_K.
\]
\end{lemma}
\end{samepage}

\begin{proof}
Fix \(N\ge1\) and \(x\in\mathbb R^n\). Let \(O\) be an open set as in
\cref{def:osc}. Choose \(z\in O\) and \(r>0\) such that
\(B(z,r)\subset O\), and put
\[
M:=\max_{y\in K}\|z-y\|.
\]
Set
\[
\mathcal I_{N,x}:=
\{\xi\in\mathcal D^N:K_\xi\cap B(x,b^{-N})\ne\varnothing\}.
\]
For each \(\xi\in\mathcal I_{N,x}\), since \(K_\xi=F_\xi(K)\), choose
\(y_\xi\in K\) such that \(F_\xi(y_\xi)\in B(x,b^{-N})\). Since
\(F_\xi\) has contraction ratio \(b^{-N}\),
\[
\|F_\xi(z)-x\|
\le b^{-N}\|z-y_\xi\|+b^{-N}
\le (M+1)b^{-N}.
\]
Iterating OSC, the sets \(F_\xi(O)\), \(\xi\in\mathcal D^N\), are pairwise
disjoint. Since \(B(z,r)\subset O\) and \(F_\xi\) is a similarity, the balls
\[
A_\xi:=F_\xi(B(z,r))=B(F_\xi(z),rb^{-N}),
\qquad \xi\in\mathcal I_{N,x},
\]
are pairwise disjoint, and the preceding bound gives
\(A_\xi\subset B(x,(M+1+r)b^{-N})\). Comparing volumes therefore yields
\[
\#\mathcal I_{N,x}\bigl(rb^{-N}\bigr)^n
\le
\bigl((M+1+r)b^{-N}\bigr)^n.
\]
Thus the counting estimate holds with
\(C_K=((M+1+r)/r)^n\). Since \(N\) and \(x\) were arbitrary, the bound is
uniform.
\end{proof}

\begin{definition}[Digit-freezing measures]
\label[definition]{def:digit-freezing}
Fix \(\mathbf a_*\in\mathcal D\). For \(P\subset\mathbb N\), define
\begin{equation}\label{eq:free-position-count}
U_P(N):=\#\bigl(\{1,\ldots,N\}\setminus P\bigr),
\qquad N\ge1,
\end{equation}
and
\begin{equation}
E(P):=\{\pi(\xi):\xi_r=\mathbf a_*\text{ for every }r\in P\}.
\end{equation}
The positions in \(P\) are \emph{prescribed}, while those in
\(\mathbb N\setminus P\) are \emph{free}; thus \(U_P(N)\) counts the free
positions up to \(N\).

To each \(P\subset\mathbb N\), associate the following measures. Let
\[
m_{\mathcal D}:=\frac{1}{\#\mathcal D}
\sum_{\mathbf a\in\mathcal D}\delta_{\mathbf a},
\]
where \(\delta_{\mathbf a}\) is the Dirac mass at \(\mathbf a\), and for
\(r\ge1\) set
\[
\nu_{P,r}:=
\begin{cases}
\delta_{\mathbf a_*},&r\in P,\\
m_{\mathcal D},&r\notin P.
\end{cases}
\]
Define
\[
\nu_P:=\bigotimes_{r=1}^{\infty}\nu_{P,r}
\quad\text{on }\mathcal D^{\mathbb N},
\qquad
\mu_P:=\pi_*\nu_P
\quad\text{on }K.
\]
We call \(\mu_P\) the \emph{digit-freezing measure} associated with \(P\).
It is supported on \(E(P)\), and in particular \(\mu_P(E(P))=1\).
\end{definition}

\begin{lemma}\label[lemma]{lem:digit-freezing}
Suppose that the common-base system satisfies OSC. Then
\[
\dim_HE(P)
\ge
\left(\liminf_{N\to\infty}\frac{U_P(N)}N\right)\dim_HK.
\]
\end{lemma}

\begin{proof}
Put
\[
h:=\liminf_{N\to\infty}\frac{U_P(N)}N.
\]
For the measures in \cref{def:digit-freezing}, every
\(\xi\in\mathcal D^N\) compatible with the prescribed positions satisfies
\(\nu_P\{\eta:\eta|_N=\xi\}=(\#\mathcal D)^{-U_P(N)}\).

If \(h=0\), the claim is immediate. Assume \(h>0\),
and fix \(0<\tau<h\). For all sufficiently large \(N\), we have
\(U_P(N)\ge\tau N\). If
\(b^{-(N+1)}<R\le b^{-N}\), then
\(B(x,R)\subset B(x,b^{-N})\). By \cref{lem:osc-counting}, at most \(C_K\)
level-\(N\) cylinders meet \(B(x,R)\). Moreover,
\[
\pi^{-1}(B(x,R))\cap\operatorname{supp}\nu_P
\subset
\bigcup_{\substack{
\xi\in\mathcal D^N\\
K_\xi\cap B(x,R)\ne\varnothing
}}
\{\eta\in\operatorname{supp}\nu_P:\eta|_N=\xi\}.
\]
The union on the right has at most \(C_K\) nonempty sets, each of
\(\nu_P\)-mass \((\#\mathcal D)^{-U_P(N)}\). 
Since \(\#\mathcal D=b^{\dim_HK}\) by Remark~\ref{rem:osc-dimension},
we have
\[
\mu_P(B(x,R))
\le C_K(\#\mathcal D)^{-U_P(N)}
\le C_K\,b^{\tau\dim_HK}R^{\tau\dim_HK}.
\]
Here the last inequality uses \(U_P(N)\ge\tau N\) and
\(R>b^{-(N+1)}\).
The mass distribution principle
\cite[Principle~4.2]{Fal03} gives
\(\dim_HE(P)\ge\tau\dim_HK\). Letting \(\tau\uparrow h\) proves the claim.
\end{proof}

\begin{remark}
Under the hypothesis of Lemma~\ref{lem:digit-freezing}, although only the lower bound is needed below, one can in fact show that equality holds.
\end{remark}

To pass to totally irrational points, we use the following property.

\begin{lemma}\label[lemma]{lem:digit-totally-irrational}
Let \(P\subset\mathbb N\) have infinitely many free positions, and let
\(\nu_P,\mu_P\) be as in \cref{def:digit-freezing}. If
\(\operatorname{aff}_{\mathbb R}(\mathcal D)=\mathbb R^n\), then every affine
hyperplane has \(\mu_P\)-measure zero; in particular, \(\mu_P\)-almost every
point is totally irrational.
\end{lemma}

\begin{proof}
Fix an affine hyperplane
\[
H=\{x\in\mathbb R^n:u\cdot x=t\},
\qquad
u\in\mathbb R^n\setminus\{0\},\quad t\in\mathbb R.
\]
Since \(\operatorname{aff}_{\mathbb R}(\mathcal D)=\mathbb R^n\), the set
\(u\cdot\mathcal D\) has at least two elements. Put
\[
\delta:=\min_{\substack{c,c'\in u\cdot\mathcal D\\c\ne c'}}|c-c'|>0.
\]
Choose \(L\ge1\) so that
\[
\frac{\operatorname{diam}(u\cdot\mathcal D)}{b^L-1}<\delta,
\]
and choose free positions
\[
r_1<r_2<\cdots,
\qquad r_{j+1}-r_j\ge L.
\]
Put
\[
S:=\{r_j:j\ge1\},
\qquad
\nu_S:=\bigotimes_{r\in S}m_{\mathcal D},
\qquad
\nu_{S^c}:=\bigotimes_{r\notin S}\nu_{P,r}.
\]
Since every position in \(S\) is free, the coordinate identification
\(\mathcal D^{\mathbb N}\simeq
\mathcal D^{S}\times\mathcal D^{S^c}\) identifies
\(\nu_P\) with \(\nu_S\otimes\nu_{S^c}\). For
\(\eta\in\mathcal D^{S^c}\), define
\[
B_\eta:=\{\zeta\in\mathcal D^S:\pi(\zeta,\eta)\in H\}.
\]
Fubini's theorem gives
\begin{equation}\label{eq:hyperplane-section-decomposition}
\mu_P(H)
=\nu_P\bigl(\pi^{-1}(H)\bigr)
=\int_{\mathcal D^{S^c}}
\nu_S(B_\eta)\,d\nu_{S^c}(\eta).
\end{equation}
Thus it remains to prove the following.

\medskip
\noindent\textbf{Claim.}
For every \(\eta\in\mathcal D^{S^c}\),
\[
\nu_S(B_\eta)=0.
\]

\smallskip
\begin{proof}[Proof of the claim]
By the choice of \(L\) and the gaps \(r_{j+1}-r_j\), for every \(j\ge1\),
\begin{equation}\label{eq:projected-tail-separation}
\begin{aligned}
\operatorname{diam}(u\cdot\mathcal D)\sum_{\ell>j}b^{-r_\ell}
&\le
\operatorname{diam}(u\cdot\mathcal D)b^{-r_j}
\sum_{q\ge1}b^{-qL}\\
&=
\frac{\operatorname{diam}(u\cdot\mathcal D)}{b^L-1}b^{-r_j}
<\delta b^{-r_j}.
\end{aligned}
\end{equation}

Fix \(\eta\in\mathcal D^{S^c}\). The equation defining \(B_\eta\) is
\begin{equation}\label{eq:section-projected-sum}
\sum_{j\ge1}(u\cdot\zeta_{r_j})b^{-r_j}
=t-\sum_{r\notin S}(u\cdot\eta_r)b^{-r}.
\end{equation}
For distinct sequences
\((c_j)_{j\ge1},(c_j')_{j\ge1}\in
(u\cdot\mathcal D)^{\mathbb N}\), put
\(k:=\min\{j:c_j\ne c_j'\}\). By
\eqref{eq:projected-tail-separation},
\[
\begin{aligned}
\left|\sum_{j\ge1}(c_j-c_j')b^{-r_j}\right|
&\ge
\delta b^{-r_k}
-\operatorname{diam}(u\cdot\mathcal D)\sum_{\ell>k}b^{-r_\ell}\\
&>0.
\end{aligned}
\]
Thus distinct projected sequences give distinct values of the left-hand side
of \eqref{eq:section-projected-sum}. Since its right-hand side is fixed by
\(\eta\), \eqref{eq:section-projected-sum} admits at most one projected
sequence. Put
\[
m:=\max_{c\in u\cdot\mathcal D}
\#\{\mathbf a\in\mathcal D:u\cdot\mathbf a=c\}<\#\mathcal D.
\]
The strict inequality holds because \(u\cdot\mathcal D\) has at least two elements.
If \(B_\eta=\varnothing\), the claim is immediate. Otherwise, let
\((c_j)_{j\ge1}\) be its projected sequence. For every \(N\ge1\),
\[
\begin{aligned}
0\le\nu_S(B_\eta)
&\le
\prod_{j=1}^N
\frac{\#\{\mathbf a\in\mathcal D:u\cdot\mathbf a=c_j\}}
{\#\mathcal D}\\
&\le
\left(\frac{m}{\#\mathcal D}\right)^N
\xrightarrow[N\to\infty]{}0.
\end{aligned}
\]
\end{proof}

The claim and \eqref{eq:hyperplane-section-decomposition} give
\(\mu_P(H)=0\).
Finally, the set of points that are not totally irrational is the countable
union
\[
\bigcup_{\mathbf p\in\mathbb Z^n\setminus\{0\}}
\ \bigcup_{p_0\in\mathbb Z}
\{x\in\mathbb R^n:\mathbf p\cdot x=-p_0\}
\]
of affine hyperplanes, so it has \(\mu_P\)-measure zero.
\end{proof}

\section{Approximation from constant-symbol blocks}
\label{constant-symbol-block-approximation}

Fix the weight \(\mathbf w=(w_1,\ldots,w_n)\) from the introduction. The next
lemma converts a constant-symbol block into the weighted approximation in
\eqref{eq:weighted-singularity}.

\begin{lemma}\label[lemma]{lem:constant-block-approximation}
There is a constant \(C_0\ge1\), depending only on \(\mathcal D\) and
\(\mathbf w\), with the following property. Suppose \(x=\pi(\xi)\in K\)
and there are integers \(j,L\ge1\) such that
\[
\xi_{j+1}=\cdots=\xi_{j+L}=\mathbf a_*;
\]
see \cref{fig:constant-block-coding}. Set \(q:=(b-1)b^j\). Then there is
\(\mathbf p\in\mathbb Z^n\) such that
\[
\max_{1\le i\le n}|qx_i-p_i|^{1/w_i}
\le C_0b^{-L/w_1}.
\]
Moreover, if \(\varepsilon>0\) and
\[
q<T<\frac{\varepsilon}{C_0}b^{L/w_1},
\]
then the same \(q\) and \(\mathbf p\) satisfy
the inequalities in \eqref{eq:weighted-singularity}.
\end{lemma}

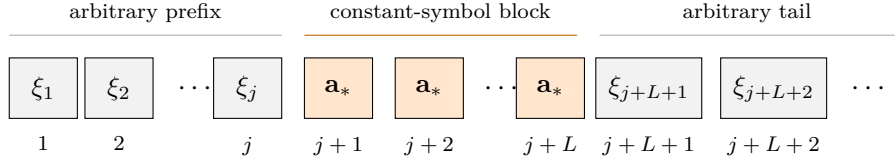
\begin{figure}[!t]
\centering
\begin{tikzpicture}[
  digit/.style={draw, minimum height=8mm, minimum width=9mm,
    inner xsep=2pt, align=center, font=\small},
  long digit/.style={digit, minimum width=14mm},
  digit index/.style={font=\scriptsize, anchor=north},
  group label/.style={font=\scriptsize, align=center}
]
\node[group label] at (1.35,1.00) {arbitrary prefix};
\draw[gray!65] (-0.45,0.70) -- (3.15,0.70);
\node[group label] at (5.25,1.00) {constant-symbol block};
\draw[orange!75!black] (3.45,0.70) -- (7.05,0.70);
\node[group label] at (9.30,1.00) {arbitrary tail};
\draw[gray!65] (7.35,0.70) -- (11.25,0.70);

\node[digit, fill=gray!10] (d1) at (0,0) {\(\xi_1\)};
\node[digit index] at ([yshift=-1mm]d1.south) {\(1\)};
\node[digit, fill=gray!10] (d2) at (1.00,0) {\(\xi_2\)};
\node[digit index] at ([yshift=-1mm]d2.south) {\(2\)};
\node at (2.00,0) {\(\cdots\)};
\node[digit, fill=gray!10] (dj) at (2.70,0) {\(\xi_j\)};
\node[digit index] at ([yshift=-1mm]dj.south) {\(j\)};

\node[digit, fill=orange!20] (da) at (3.90,0) {\(\mathbf a_*\)};
\node[digit index] at ([yshift=-1mm]da.south) {\(j+1\)};
\node[digit, fill=orange!20] (db) at (5.10,0) {\(\mathbf a_*\)};
\node[digit index] at ([yshift=-1mm]db.south) {\(j+2\)};
\node at (6.05,0) {\(\cdots\)};
\node[digit, fill=orange!20] (dc) at (6.70,0) {\(\mathbf a_*\)};
\node[digit index] at ([yshift=-1mm]dc.south) {\(j+L\)};

\node[long digit, fill=gray!10] (dt1) at (8.00,0)
  {\(\xi_{j+L+1}\)};
\node[digit index] at ([yshift=-1mm]dt1.south) {\(j+L+1\)};
\node[long digit, fill=gray!10] (dt2) at (9.65,0)
  {\(\xi_{j+L+2}\)};
\node[digit index] at ([yshift=-1mm]dt2.south) {\(j+L+2\)};
\node at (10.90,0) {\(\cdots\)};
\end{tikzpicture}
\caption{A constant-symbol block as in
\cref{lem:constant-block-approximation}. The orange-shaded digits are fixed,
while the prefix and tail are arbitrary.}
\label{fig:constant-block-coding}
\end{figure}
\FloatBarrier

\begin{proof}
Set
\[
\Delta_{\mathcal D}
:=
\max_{\mathbf a,\mathbf a'\in\mathcal D}
\|\mathbf a-\mathbf a'\|,
\qquad
C_0:=\Delta_{\mathcal D}^{1/w_n}.
\]
Since \(\mathcal D\subset\mathbb Z^n\) has at least two elements,
\(\Delta_{\mathcal D}\ge1\), so \(C_0\ge1\).

Extend the prefix \((\xi_1,\ldots,\xi_j)\) by repeating
\(\mathbf a_*\) forever. The resulting point lies in \(K\) and is the
rational vector
\[
\frac{\mathbf p}{q}
:=\sum_{r=1}^j\xi_rb^{-r}
  +\sum_{r=j+1}^{\infty}\mathbf a_*b^{-r},
\]
where \(q=(b-1)b^j\) and
\[
\mathbf p
=(b-1)\sum_{r=1}^j\xi_rb^{j-r}+\mathbf a_*
\in\mathbb Z^n.
\]
Let
\(y:=\pi(\xi_{j+L+1},\xi_{j+L+2},\ldots)\in K\)
be the tail after the constant-symbol block. Since \(x\) and
\(\mathbf p/q\) have the same first \(j+L\) digits,
\[
x-\frac{\mathbf p}{q}
=\sum_{r=j+L+1}^{\infty}(\xi_r-\mathbf a_*)b^{-r}
=b^{-(j+L)}\left(y-\frac{\mathbf a_*}{b-1}\right).
\]
Multiplying by \(q=(b-1)b^j\) gives
\[
qx-\mathbf p
=(b-1)b^{-L}
\left(y-\frac{\mathbf a_*}{b-1}\right).
\]
Moreover,
\[
\left\|y-\frac{\mathbf a_*}{b-1}\right\|
\le\sum_{r=1}^\infty\|\xi_{j+L+r}-\mathbf a_*\|b^{-r}
\le
\Delta_{\mathcal D}\sum_{r=1}^\infty b^{-r}
=\frac{\Delta_{\mathcal D}}{b-1}.
\]
It follows that
\[
|qx_i-p_i|\le\Delta_{\mathcal D}b^{-L}
\qquad(1\le i\le n).
\]
For each \(i\), the inequalities \(w_n\le w_i\le w_1\) and \(b>1\) give
\[
|qx_i-p_i|^{1/w_i}
\le
\Delta_{\mathcal D}^{1/w_i}b^{-L/w_i}
\le
\Delta_{\mathcal D}^{1/w_n}b^{-L/w_1}
=C_0b^{-L/w_1}.
\]
Taking the maximum over \(i\) proves the first assertion. Finally, if
\(q<T<(\varepsilon/C_0)b^{L/w_1}\), then \(0<q<T\) and
\[
C_0b^{-L/w_1}<\frac{\varepsilon}{T},
\]
which proves the second assertion.
\end{proof}

\section{Block construction}\label{block-construction}

Fix \(C_0\) as in \cref{lem:constant-block-approximation}. The following
proposition provides the block locations and lengths used in the construction.
Its proof is given in \cref{parameter-estimates}.

\begin{proposition}\label[proposition]{prop:parameters}
For any fixed \(\alpha>w_1/(1-w_1)\), there exist sequences of positive
integers \(j_k=j_k(\alpha)\) and \(L_k=L_k(\alpha)\) such that the sets and
denominators
\[
B_k:=\{j_k+1,\ldots,j_k+L_k\},
\qquad
q_k:=(b-1)b^{j_k},
\qquad
P_\alpha:=\bigcup_{k\ge1}B_k
\]
satisfy
\[
j_k+L_k<j_{k+1}
\qquad(k\ge1).
\]
Moreover, for every \(\varepsilon>0\), there exists \(k_0\) such that
\begin{equation}\label{eq:block-cover}
(q_{k_0},\infty)
\subset
\bigcup_{k\ge k_0}
\left(q_k,\frac{\varepsilon}{C_0}b^{L_k/w_1}\right).
\end{equation}
Finally, with the notation of \eqref{eq:free-position-count}, set
\[
U_\alpha(N):=U_{P_\alpha}(N).
\]
Then
\begin{equation}\label{eq:block-density}
\liminf_{N\to\infty}\frac{U_\alpha(N)}N
=
h_{\mathbf w}(\alpha)
:=
\frac{\alpha(1-w_1)-w_1}{(\alpha-w_1)(1+\alpha)}.
\end{equation}
\end{proposition}

With the notation of \cref{def:digit-freezing}, define

\[
E_\alpha
:=E(P_\alpha)
=
\left\{
\pi(\xi):
\begin{array}{l}
\xi\in\mathcal D^{\mathbb N},\\
\xi_{j_k+r}=\mathbf a_*
\quad
(k\ge1,\ 1\le r\le L_k)
\end{array}
\right\}.
\]
Thus \(E_\alpha\) is obtained by fixing the digits on the blocks
\(B_k\), while all remaining digits are free.

\begin{corollary}\label[corollary]{cor:singularity}
One has

\[
E_\alpha\subset\operatorname{Sing}(\mathbf w)\cap K.
\]
\end{corollary}

\begin{proof}
Fix \(\varepsilon>0\), and choose \(k_0\) as in
\eqref{eq:block-cover}. If \(T>q_{k_0}\), then there is
\(k\ge k_0\) such that
\[
q_k<T<\frac{\varepsilon}{C_0}b^{L_k/w_1}.
\]
For every \(x\in E_\alpha\),
the \(k\)-th prescribed block and
\cref{lem:constant-block-approximation} give
\(\mathbf p_k\in\mathbb Z^n\) such that
\[
0<q_k<T,
\qquad
\max_{1\le i\le n}|q_kx_i-p_{i,k}|^{1/w_i}
<\frac{\varepsilon}{T}.
\]
Hence \(x\in\operatorname{Sing}(\mathbf w)\).
\end{proof}

\begin{corollary}\label[corollary]{cor:dimension}
Suppose that the common-base system satisfies OSC. Then
\[
\dim_HE_\alpha\ge h_{\mathbf w}(\alpha)\dim_HK.
\]
\end{corollary}

\begin{proof}
This follows from \eqref{eq:block-density} and
\cref{lem:digit-freezing}, since \(U_\alpha=U_{P_\alpha}\).
\end{proof}

\begin{proof}[Proof of Theorem~\ref{thm:intro-common-base}]
By \cref{cor:singularity,cor:dimension}, the set \(E_\alpha\) is
contained in \(\operatorname{Sing}(\mathbf w)\cap K\) and has dimension at
least \(h_{\mathbf w}(\alpha)\dim_HK\) for every
\(\alpha>w_1/(1-w_1)\). Differentiation gives

\[
h_{\mathbf w}'(\alpha)
=
-\frac{
\alpha\bigl((1-w_1)\alpha-2w_1\bigr)
}{
(\alpha-w_1)^2(1+\alpha)^2
}.
\]

The unique critical point in the admissible interval is
\(\alpha_*:=2w_1/(1-w_1)\). Since the function tends to zero at both ends of
the interval, this point gives its maximum, and
\(h_{\mathbf w}(\alpha_*)=((1-w_1)/(1+w_1))^2\). Put
\(E=E_{\alpha_*}\) and \(\mu:=\mu_{P_{\alpha_*}}\), as in
\cref{def:digit-freezing}.
Since
\(h_{\mathbf w}(\alpha_*)>0\), there are infinitely many free positions, and
\cref{lem:digit-totally-irrational} shows that \(\mu\)-almost every point is
totally irrational. Since \(E\subset\operatorname{Sing}(\mathbf w)\), the set
\(E\cap\operatorname{Sing}^*(\mathbf w)\) has full \(\mu\)-measure. For every
\(s<h_{\mathbf w}(\alpha_*)\dim_HK\), the ball estimate in the proof of
\cref{lem:digit-freezing} remains valid after restricting \(\mu\) to this set.
Consequently,
\[
\dim_H\bigl(E\cap\operatorname{Sing}^*(\mathbf w)\bigr)
\ge h_{\mathbf w}(\alpha_*)\dim_HK.
\]
Since \(E\subset K\), this proves the theorem.
\end{proof}

\begin{proof}[Proof of Theorem~\ref{thm:cantor-square}]
The Cantor square \(C^2\) is the attractor of the base-\(3\) system with
digit set \(\{0,2\}^2\), which affinely spans \(\mathbb R^2\). It satisfies OSC with
\(O=(0,1)^2\), and \(\dim_H(C^2)=2\log2/\log3\). The weighted assertion is
therefore the special case \(b=3\), \(\mathcal D=\{0,2\}^2\) of
\cref{thm:intro-common-base}. The unweighted statement follows from
\(w_1=w_2=1/2\).
\end{proof}

\section{Block parameter estimates}\label{parameter-estimates}

\begin{proof}[Proof of Proposition~\ref{prop:parameters}]
Fix \(\alpha>w_1/(1-w_1)\).

\medskip
\noindent\textbf{Step 1: Choice, growth, and separation.}
Put
\[
\rho:=\frac{\alpha}{w_1},\qquad
\rho_0:=\frac{1+\rho}{2}.
\]
Since both \(\rho-\rho_0\) and
\(\alpha(1-w_1)/w_1-1\) are positive, choose \(j_1\) so large that, for
every \(j\ge j_1\),
\[
\min\left\{
\rho-\rho_0,\frac{\alpha(1-w_1)}{w_1}-1
\right\}j
\ge\sqrt j+3.
\]
Define recursively
\begin{equation}\label{eq:block-recursion}
L_k:=\lceil\alpha j_k\rceil,
\qquad
j_{k+1}:=\left\lfloor\frac{L_k}{w_1}\right\rfloor
-\lceil\sqrt{j_k}\rceil.
\end{equation}
The ceiling and floor bounds in \eqref{eq:block-recursion} give
\begin{equation}\label{eq:parameter-bounds}
\begin{aligned}
\alpha j_k&\le L_k<\alpha j_k+1,\\
\sqrt{j_k}
&\le \frac{L_k}{w_1}-j_{k+1}<\sqrt{j_k}+2,\\
\rho j_k-\sqrt{j_k}-2&<j_{k+1}
<\rho j_k+\frac1{w_1}.
\end{aligned}
\end{equation}

Suppose that \(j_k\ge j_1\). Then the first condition in the choice of
\(j_1\), together with \eqref{eq:parameter-bounds}, gives
\begin{equation}\label{eq:uniform-growth}
\begin{aligned}
j_{k+1}
&>\rho j_k-\sqrt{j_k}-2\\
&=\rho_0j_k+
  \bigl((\rho-\rho_0)j_k-\sqrt{j_k}-2\bigr)
\ge \rho_0j_k+1.
\end{aligned}
\end{equation}
Since \(\rho_0>1\), induction shows that \eqref{eq:uniform-growth} holds for
every \(k\) and that \(j_k\ge\rho_0^{k-1}j_1\). In particular,
\(j_k\to\infty\). Moreover,
\begin{align*}
j_{k+1}-j_k-L_k
&>L_k\left(\frac1{w_1}-1\right)-j_k-\sqrt{j_k}-2\\
&\ge
\left(\frac{\alpha(1-w_1)}{w_1}-1\right)j_k
-\sqrt{j_k}-2\\
&\ge1.
\end{align*}
Hence \(j_k+L_k<j_{k+1}\).

\medskip
\noindent\textbf{Step 2: Covering all large \(T\).}
Set
\[
\varepsilon_k
:=2(b-1)C_0b^{-(L_k/w_1-j_{k+1})}.
\]
By the middle estimate in \eqref{eq:parameter-bounds},
\[
0<\varepsilon_k
\le 2(b-1)C_0b^{-\sqrt{j_k}}\longrightarrow0.
\]
Furthermore, a direct cancellation gives
\[
\begin{aligned}
\frac{\varepsilon_k}{C_0}b^{L_k/w_1}
&=2(b-1)b^{-(L_k/w_1-j_{k+1})}b^{L_k/w_1}\\
&=2(b-1)b^{j_{k+1}}
=2q_{k+1}>q_{k+1}.
\end{aligned}
\]
Given \(\varepsilon>0\), choose \(k_0\) so that
\(\varepsilon_k<\varepsilon\) for \(k\ge k_0\). If \(T>q_{k_0}\), choose
\(k\ge k_0\) such that \(q_k<T\le q_{k+1}\), which is possible since
\(q_k\to\infty\). Then
\[
q_k<T\le q_{k+1}
<\frac{\varepsilon_k}{C_0}b^{L_k/w_1}
<\frac{\varepsilon}{C_0}b^{L_k/w_1}.
\]
Thus \(T\) belongs to the \(k\)-th interval in \eqref{eq:block-cover}.

\medskip
\noindent\textbf{Step 3: Density of free positions.}
Let \(N_k:=j_k+L_k\). Since the prescribed blocks are disjoint,
\(U_\alpha\) is constant on \([j_k,N_k]\) and increases with slope one on
\([N_k,j_{k+1}]\). Consequently, \(U_\alpha(N)/N\) decreases on the first
interval and increases on the second. Its successive local minima therefore
occur at \(N_k\), as illustrated in \cref{fig:free-position-count}, and
\begin{equation}\label{eq:endpoint-density}
\liminf_{N\to\infty}\frac{U_\alpha(N)}N
=
\liminf_{k\to\infty}
\frac{j_k-\sum_{h<k}L_h}{j_k+L_k}.
\end{equation}

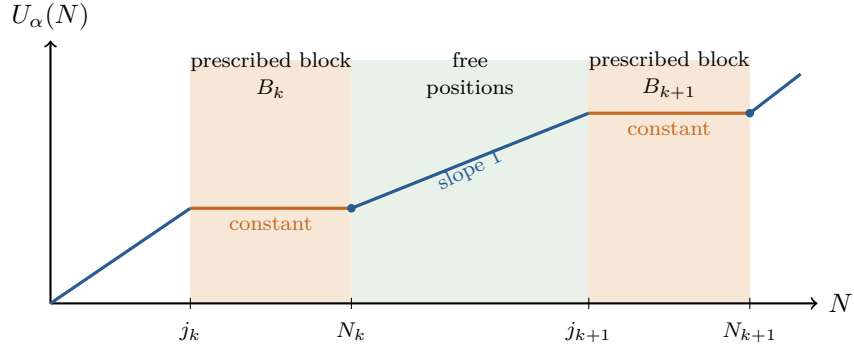
\begin{figure}[!t]
\centering
\begin{tikzpicture}[
  x=1.12cm,
  y=0.74cm,
  axis label/.style={font=\small},
  interval label/.style={font=\scriptsize, align=center},
  tick label/.style={font=\scriptsize, anchor=north}
]
\definecolor{PrescribedBlock}{RGB}{247,231,214}
\definecolor{FreeInterval}{RGB}{232,241,234}
\definecolor{CountBlue}{RGB}{41,92,145}
\definecolor{BlockOrange}{RGB}{196,105,36}

\def\xA{1.65}
\def\xB{3.55}
\def\xC{6.35}
\def\xD{8.25}

\fill[PrescribedBlock] (\xA,0.35) rectangle (\xB,4.70);
\fill[FreeInterval] (\xB,0.35) rectangle (\xC,4.70);
\fill[PrescribedBlock] (\xC,0.35) rectangle (\xD,4.70);

\draw[->,thick] (0,0.35) -- (9.05,0.35)
  node[right,axis label] {\(N\)};
\draw[->,thick] (0,0.35) -- (0,5.05)
  node[above,axis label] {\(U_\alpha(N)\)};

\draw[CountBlue,very thick] (0,0.35) -- (\xA,2.05);
\draw[BlockOrange,very thick] (\xA,2.05) -- (\xB,2.05);
\draw[CountBlue,very thick] (\xB,2.05) -- (\xC,3.75);
\draw[BlockOrange,very thick] (\xC,3.75) -- (\xD,3.75);
\draw[CountBlue,very thick] (\xD,3.75) -- (8.85,4.45);

\foreach \x/\lab in {
  \xA/{\(j_k\)},
  \xB/{\(N_k\)},
  \xC/{\(j_{k+1}\)},
  \xD/{\(N_{k+1}\)}
}{
  \draw (\x,0.25) -- (\x,0.45);
  \node[tick label] at (\x,0.20) {\lab};
}

\fill[CountBlue] (\xB,2.05) circle (1.6pt);
\fill[CountBlue] (\xD,3.75) circle (1.6pt);

\node[interval label] at (2.60,4.43)
  {prescribed block\\\(B_k\)};
\node[interval label] at (4.95,4.43)
  {free\\positions};
\node[interval label] at (7.30,4.43)
  {prescribed block\\\(B_{k+1}\)};
\node[interval label,BlockOrange] at (2.60,1.77)
  {constant};
\node[interval label,CountBlue,rotate=26] at (4.95,2.73)
  {slope \(1\)};
\node[interval label,BlockOrange] at (7.30,3.47)
  {constant};
\end{tikzpicture}
\caption{The count is constant on prescribed blocks and has slope one
elsewhere, so \(U_\alpha(N)/N\) is minimized at \(N_k\).}
\label{fig:free-position-count}
\end{figure}
\FloatBarrier

By \eqref{eq:parameter-bounds} and the fact that \(j_k\to\infty\),
\[
\frac{L_k}{j_k}\longrightarrow\alpha,
\qquad
\frac{j_{k+1}}{j_k}\longrightarrow\rho.
\]
Hence
\[
\frac{L_k}{j_{k+1}-j_k}
=
\frac{L_k/j_k}{j_{k+1}/j_k-1}
\longrightarrow
\frac{\alpha}{\rho-1}.
\]
Since \(j_k\) is strictly increasing and tends to infinity, the
Stolz--Ces\`aro theorem gives
\[
\lim_{k\to\infty}\frac{\sum_{h<k}L_h}{j_k}
=
\lim_{k\to\infty}\frac{L_k}{j_{k+1}-j_k}
=
\frac{\alpha}{\rho-1}.
\]
Consequently,
\[
\frac{U_\alpha(N_k)}{N_k}
=
\frac{j_k-\sum_{h<k}L_h}{j_k+L_k}
\longrightarrow
\frac{1-\alpha/(\rho-1)}{1+\alpha}
=
\frac{\alpha(1-w_1)-w_1}{(\alpha-w_1)(1+\alpha)}.
\]

Together with \eqref{eq:endpoint-density}, this proves
\eqref{eq:block-density}.
\end{proof}


\begin{thebibliography}{KMWW25}

\bibitem[AG26]{AG26}
Gaurav Aggarwal and Anish Ghosh, \emph{On the packing dimension of weighted
  singular matrices on fractals}, Advances in Mathematics \textbf{493} (2026),
  110931.

\bibitem[BCC19]{BCC19}
Yann Bugeaud, Yitwah Cheung, and Nicolas Chevallier, \emph{Hausdorff dimension
  and uniform exponents in dimension two}, Mathematical Proceedings of the
  Cambridge Philosophical Society \textbf{167} (2019), no.~2, 249--284.

\bibitem[CC16]{CC16}
Yitwah Cheung and Nicolas Chevallier, \emph{Hausdorff dimension of singular
  vectors}, Duke Mathematical Journal \textbf{165} (2016), no.~12, 2273--2329.

\bibitem[Che11]{Che11}
Yitwah Cheung, \emph{Hausdorff dimension of the set of singular pairs}, Annals
  of Mathematics \textbf{173} (2011), no.~1, 127--167.

\bibitem[DFSU24]{DFSU24}
Tushar Das, Lior Fishman, David Simmons, and Mariusz Urba{\'n}ski, \emph{A
  variational principle in the parametric geometry of numbers}, Advances in
  Mathematics \textbf{437} (2024), 109435.

\bibitem[DT24]{DT24}
Shreyasi Datta and Nattalie Tamam, \emph{On weighted singular vectors for
  multiple weights}, 2024, arXiv:2409.17105v2.

\bibitem[Fal03]{Fal03}
Kenneth~J. Falconer, \emph{Fractal geometry: Mathematical foundations and
  applications}, second ed., John Wiley \& Sons, Chichester, 2003.

\bibitem[Hut81]{Hut81}
John~E. Hutchinson, \emph{Fractals and self-similarity}, Indiana University
  Mathematics Journal \textbf{30} (1981), no.~5, 713--747.

\bibitem[Kha20]{Kha20}
Osama Khalil, \emph{Singular vectors on fractals and projections of
  self-similar measures}, Geometric and Functional Analysis \textbf{30} (2020),
  no.~2, 482--535.

\bibitem[Khi26]{Khi26}
Aleksandr Khintchine, \emph{\"Uber eine klasse linearer diophantischer
  approximationen}, Rendiconti del Circolo Matematico di Palermo \textbf{50}
  (1926), 170--195.

\bibitem[KKLM17]{KKLM17}
Shirali Kadyrov, Dmitry Kleinbock, Elon Lindenstrauss, and Gregory~A. Margulis,
  \emph{Singular systems of linear forms and non-escape of mass in the space of
  lattices}, Journal d'Analyse Math\'{e}matique \textbf{133} (2017), 253--277.

\bibitem[KMW21]{KMW21}
Dmitry Kleinbock, Nikolay Moshchevitin, and Barak Weiss, \emph{Singular vectors
  on manifolds and fractals}, Israel Journal of Mathematics \textbf{245}
  (2021), no.~2, 589--613.

\bibitem[KMWW25]{KMWW25}
Dmitry Kleinbock, Nikolay Moshchevitin, Jacqueline~M. Warren, and Barak Weiss,
  \emph{Singularity, weighted uniform approximation, intersections and rates},
  Compositio Mathematica \textbf{161} (2025), no.~11, 2990--3016.

\bibitem[KP24]{KP24}
Taehyeong Kim and Jaemin Park, \emph{On a lower bound of {Hausdorff} dimension
  of weighted singular vectors}, Mathematika \textbf{70} (2024), no.~3, e12252.

\bibitem[LSST20]{LSST20}
Lingmin Liao, Ronggang Shi, Omri~N. Solan, and Nattalie Tamam, \emph{Hausdorff
  dimension of weighted singular vectors in {$\mathbb R^2$}}, Journal of the
  European Mathematical Society \textbf{22} (2020), no.~3, 833--875.

\bibitem[Sch22]{Sch22}
Johannes Schleischitz, \emph{Metric results on inhomogeneously singular
  vectors}, 2022, arXiv:2201.01527, version 3.

\bibitem[SY24]{SY24}
Nimish~A. Shah and Pengyu Yang, \emph{An upper bound of the {Hausdorff}
  dimension of singular vectors on affine subspaces}, Transactions of the
  American Mathematical Society, Series B \textbf{11} (2024), no.~36,
  1249--1265.

\end{thebibliography}
\end{document}